\documentclass[a4paper,12pt]{amsart}
\usepackage{amsmath}
\usepackage{mathtools}
\usepackage{extsizes}
\usepackage[utf8]{inputenc}
\usepackage{amssymb, latexsym}
\usepackage{xcolor}
\usepackage{hyperref}
\usepackage{ulem}
\usepackage{amsthm}
 
\newtheorem{theorem}{Theorem}[section]
\newtheorem{corollary}{Corollary}[theorem]
\newtheorem{lemma}[theorem]{Lemma}
\newtheorem{proposition}[theorem]{Proposition}
\newtheorem{remark}{Remark}
\newtheorem{example}{Example}
\newtheorem{definition}{Definition}

\newcommand{\bs}{\boldsymbol}
\newcommand{\tb}{\textbf}
\newcommand{\mr}{\mathrm}

\title[Group of isometries]{Fixed points and normal automorphisms of the unit ball of  bounded operators on $\mathbb{C}^n$}
\author[ R. Aggarwal, K. Gongopadhyay and M. M. Mishra  ]{Rachna Aggarwal, Krishnendu Gongopadhyay and Mukund Madhav Mishra   }
\address{Department of Mathematics, University of Delhi, Delhi, India}
\email{rachna2389@gmail.com}

\address{Department of Mathematical Sciences, Indian Institute of Science Education and Research (IISER) Mohali, Knowledge City,
S.A.S. Nagar, Sector 81, P.O. Manauli 140306, India}
\email{krishnendu@iisermohali.ac.in}
\address{Department of Mathematics, Hans Raj College, University of Delhi, Delhi, India}
\email{mukund@hrc.du.ac.in}
\numberwithin{equation}{section}

\date{\today}
\begin{document}
 
\begin{abstract}
  We  examine the group of isometries of the open unit ball of a complex Banach space of certain bounded linear operators  equipped with the Carath\'eodory  metric. Therein we obtain a characterization of the normal isometries in terms of their special type of fixed points. 
\end{abstract}
\keywords{unit ball,  Caratan domain, automorphism group, normal isometry, bounded linear operators}
\subjclass[2020]{Primary 32M15; Secondary 47B02, 47B15, 47B91}
\maketitle
\section{Introduction}

Classical group $U(n,1)$ has been extensively studied in the literature from various perspectives. For instance, in \cite{GR}, \cite{GOL}, \cite{PA}, \cite{GON} and many more. A natural generalization of the unitary group $U(n,1)$  is the pseudo-unitary group $U(p,q)$. For its definition and details, see  \cite{NE}, for instance. It exhibits   substantial contrast in its formulative theory due to the presence of a higher rank component. These groups have been explored at certain places like \cite{NE} \cite{GP}, \cite{PO}, \cite{MO}, \cite{MU}, \cite{OS} and \cite{MA}. Some infinite dimensional versions have  also been  discussed, e.g. \cite{FR}, \cite{HA},  \cite{GW} and \cite{KW}. Also Popescu in \cite{Po} has studied hyperbolic geometry on non commutative unit balls in infinite dimensional set up.

In this article,  we consider certain infinite dimensional extension of these pseudo-unitary groups and investigate one of their significant subclass from the fixed point classification perspective. The subclass under consideration is the class of normal elements.  If we consider  holomorphic isometry  on the open unit ball of  the complex Banach space of bounded linear operators between Hilbert spaces then it can be represented by a linear operator on a Hilbert space.   \textit{\textit{We reserve the phrases \lq normal isometry\rq \,\, and \lq non-unitary normal isometry\rq \,\, for an isometry whose linear representation is normal and non-unitary normal as an operator.}} A bounded linear operator on a Hilbert space is called normal if it commutes with its Hilbert adjoint and unitary if its Hilbert adjoint coincides with its inverse. An operator is non-unitary normal if it is not unitary and a normal operator. The central result of this article is Theorem \ref{y} which characterizes a non-unitary normal isometry on the basis of a special class of its fixed points.  Before describing our main result, we make  the following set up below. \\

Let \(K\) and \(H\) be complex Hilbert spaces and $B(K,H)$ be the complex Banach space of all bounded linear operators from $K$ to $H$ equipped with the operator norm. If $K=H$, this will be abridged to $B(H)$. Let \(\mathcal{B}\) be the  open unit ball   of $B(K,H)$ equipped with the Carath\'eodory metric. It is an infinite dimensional extension of the Cartan domain of type 1.   Let $U(H)$ denote the group of unitary operators on $H$.  The holomorphic automorphisms of $\mathcal B$ has an action on the unit ball of $\mathcal B(K,H)$ by operator-valued M\"obius transformations.   This is an operator valued analogue of the classical  M\"obius transformations acting on the Poincar\'e disc.  The automorphism group of $\mathcal B$ as operator-valued linear fractional transformations was determined  independently  by Greenfield and Wallach \cite{GW}, Harris  \cite{HA} and Kaup \cite{KW} respectively in different contexts. The general form of a holomorphic automorphism of $\mathcal{B}$ also follows from these works.  In \cite{FR},  Franzoni used a different approach that works also in infinite dimensions  and provided a linear representation of the group of holomorphic automorphisms.   Considering this linear group as our framework, we mainly work with the assumption $K=\mathbb{C}^n$.  In the following subsection, we describe the main result of this article. 

\subsection{Main results of the article}

We begin by providing the machinery which is needed to understand the statement of the main result. We shall follow Franzoni, \cite{FR} in our treatment to recall the basics.  Let 
$\mathcal{G}$ denote the group of all bijective linear operators defined on $H \oplus K$ which leave  the Hermitian form $\mathcal S$ invariant, where 
$\mathcal{S}:(H \oplus K) \times (H \oplus K) \rightarrow \mathbb{C}$ is defined as
\[\mathcal{S}\big((h_1,k_1),(h_2,k_2)\big)=\left<h_1,h_2\right>-\left<k_1,k_2\right>.\]

\medskip
\textit{From now on, throughout this subsection, {$K=\mathbb{C}^n$}  unless stated otherwise and  the sets $\{1,\,2,\,...,\,l\}$ and $\{1,\,2,\,...,\,k_i\}$ will be denoted by the symbols  $I$ and $J_i$ respectively where $i \in I$.} 
 
 \medskip 
A general element of $\mathcal{G}$ is $T=\left[ {\begin{array}{cc}
   BU & CV \\
   C^*U & EV \\
  \end{array} } \right]$ where
  \begin{itemize}
  \item $U \in U(H)$, $V \in U(\mathbb{C}^n)$, $B \in B(H)$, $E \in B(\mathbb{C}^n)$ and $C \in B(\mathbb{C}^n,H)$. 

\medskip 
  \item $E$ decomposes $\mathbb{C}^n$ orthogonally into eigenspaces as $\mathbb{C}^n=\overset{l}{\underset{i=1}{\oplus}}{K_i}\oplus {K'}^{\perp}$ where $K'=\overset{l}{\underset{i=1}{\oplus}}{K_i}$.  $E \restriction_{K_i}=a_i\,I$ and $E \restriction_{{K'}^{\perp}}=I$,  $a_i > 1$, $i \in I$.\\  Let $dim\,(K')=k$ and $dim\,(K_i)=k_i$, hence $\sum\limits_{i=1}^{l}{k_i}=k$. 
  Let  $\beta_i=\{e_{i_j},\,\,j \in J_i\} $
   be an orthonormal basis of $K_i$, $i \in I$.

\medskip 
   \item  $C(e_{i_j})=\xi_{i_j}$  where $\xi_{i_j} (\neq 0) \in H$ for every $i\in I$ and $j\in J_i$ and distinct ${\xi_{i_j}}'s$  are mutually orthogonal making $C\restriction_{K'}$ a bijective operator and $C\restriction_{{K'}^{\perp}}\equiv 0$. Hence $dim\,(ran\,C)=k=dim\,(K')$ where $ran\,C=$ Range of $C$.

\medskip 
   \item $\|\xi_{i_j}\|=\|\xi_{i_{j'}}\|$ for an arbitrarily fixed $i\in I$ and for all $j,\,j' \in J_i.$ Let $\|\xi_{i_j}\|=\|\xi_{i_{j'}}\|=\delta_i$ giving  
   ${a_i}^2=1+{\delta_i}^2,\,\,i \in I.$

\medskip 
    \item $C^{*}(\xi_{i_j})={\delta_i}^2\,e_{i_j}$ and  $C^*\equiv 0$ on ${(ran\,C)}^{\perp}$, for every $i \in I$ and $j \in J_i$.

\medskip 
   \item Let $M_i=span\,\{\xi_{i_j},\,\,j\in J_i\}$. Then $ran\,C=\overset{l}{\underset{i=1}{\oplus}}{M_i}$,  $dim\,(M_i)=dim\,(K_i)\\=k_i$, $i \in I$.  The operator $B$ decomposes $H$ orthogonally into eigenspaces  as $H=\overset{l}{\underset{i=1}{\oplus}}{M_i}\oplus {(ran\,C)}^{\perp}$ where $B \restriction_{M_i}=a_i\,I$  and $B \restriction_{{(ran\,C)}^{\perp}}=I$. 
 Let  us denote an orthogonal basis $\{\xi_{i_j},\,\,j \in J_i\} $  of $M_i$, $i \in I$ by ${\beta_i}'$.
 
  \end{itemize}
    \begin{proposition}\label{P1}
Let $K$ be any arbitrary Hilbert space.  An isometry  $T=
 \left[ {\begin{array}{cc}
   BU & CV \\
   C^*U & EV \\
  \end{array} } \right] \in \mathcal{G}$ is normal if and only if $UB=BU$, $UC=CV$ and $VE=EV$.
 \end{proposition}
 \begin{proposition} \label{z}
    For $K=\mathbb{C}^n$, unitary elements in $\mathcal{G}$ are of the form $\left[ {\begin{array}{cc}
   U & 0 \\
   0 & V \\
  \end{array} } \right]$ where $U\in U(H)$ and $V \in U(\mathbb{C}^n)$.
  \end{proposition}
  \begin{proposition}\label{h}
 For $K=\mathbb{C}^n$, Let $T=\left[ {\begin{array}{cc}
   BU & CV \\
   C^*U & EV \\
  \end{array} } \right] \in \mathcal{G}$. Then $T$ is non-unitary normal if and only if 
  \begin{align}
U(M_i)&=M_i ,\notag\\
V(K_i)&=K_i ,\,\,\text{and}\notag\\
 {[U\restriction_{M_i}]}_{{\beta_i}'}&={[V\restriction_{K_i}]}_{{\beta_i}}. \label{eq:36}  
  \end{align}
   \end{proposition}
\begin{definition}\label{d1}
An operator $F \in B(\mathbb{C}^n,H)$ will be called a \textbf{fixed point} of $T=\left[ {\begin{array}{cc}
   BU & CV \\
   C^*U & EV \\
  \end{array} } \right] \in \mathcal{G}$ if there exists a basis $\{\bs{\mr{z_r}}, \,\,r=1,\,2,\,...,\,n\}$ of $\mathbb{C}^n$ such that $\left[{\begin{array}{c}
  F(\boldsymbol{\mr{z}_r})\\
  \boldsymbol{\mr{z}_r}\\
  \end{array}}\right],\,\,\,\,r=1,\,2,\,...,\,n$ are eigenvectors of $T$.  
 \end{definition} 
\begin{definition}\label{d2}
The functions $F_\theta \in B(\mathbb{C}^n,H)$ defined below will be called of \textbf{generic} type.\\
 Let $S=\{(\epsilon_m)_{m=1}^{q},\,\,\epsilon_m\in \{ 1,-1\}\}$ be a collection of $q$-tuples where $q \leq n$. For $\theta=(\epsilon_m)_{m=1}^{q} \in S$, define $F_\theta \in B(\mathbb{C}^n,H)$ as 
   \begin{align*}
      F_\theta(\bs{\mr{z}_m})&=\epsilon_m\, y_m,\,\,\,\,m=1,\,2,\,...,\,q, \\
      F_\theta(\bs{\mr{z}_s})&=0, \,\,\,\,s=q+1,\,...,\,n 
  \end{align*}
  where $\{\bs{\mr{z}_r},\,\,r=1,\,2,\,...,\,n\}$ is an orthonormal basis of $\mathcal{C}^n$ and $\{y_m,\,\,m=1,\,2,\,...,\,q\}$ are non-zero vectors in $H$. \\
  If $\left[{\begin{array}{c}
  F_\theta(\boldsymbol{\mr{z}_r})\\
  \boldsymbol{\mr{z}_r}\\
  \end{array}}\right],\,\,\,\,r=1,\,2,\,...,\,n,\,\theta \in S$ are eigenvectors of $T$ then in view of Definition \ref{d1}, ${F_\theta}$'s are called \textbf{fixed points of generic type.}
\end{definition}
Note that in Definition \ref{d2}, $|S|=2^q$. Hence $|\{F_\theta,\,\,\theta \in S\}|=2^q$.
Now we state the main result.
\begin{theorem}\label{y}
   Let an isometry $T=\left[ {\begin{array}{cc}
   BU & CV \\
   C^*U & EV \\
  \end{array} } \right] \in \mathcal{G}$ be such that $dim\,(ran\\\,C)=k$. Then $T$ is non-unitary normal if and only if it possesses exactly $2^k$ fixed points of the generic type. The generic fixed points lie on the boundary of the unit ball ${\mathcal{B}}$.
   \end{theorem}
  \textit{Structure of the article.} Section 2 reviews the work of \cite{HA} and \cite{FR} by discussing the general form of a holomorphic automorphism of $\mathcal{B}$ and its linear representation. Section 3 discusses the general form of an isometry in $\mathcal{G}$ along with the description of its normal, self-adjoint and unitary elements for a general Hilbert space $K$. Section 4 investigates general form, normal and unitary elements  under the finite dimensionality assumption on $K$. Section 5 focuses on normal isometries and explores their spectral properties along with their fixed point characterization. We prove our main result in this section.
  
  \section{Holomorphic automorphisms of $\mathcal{B}$}

As mentioned above, let $\mathcal{B}$ be the open unit ball in $B(K,H)$. 
A biholomorphic map or a holomorphic  automorphism of $\mathcal{B}$ is a bijective holomorphic mapping from $\mathcal{B}$ to $\mathcal{B}$ having holomorphic inverse. Let $\overline{\mathcal{B}}$ denote the closure of  $\mathcal{B}$ in the operator norm  and $\partial{\mathcal{B}}$,  the boundary of $\overline{\mathcal{B}}$ in $B(K,H)$.  For $T \in B(H)$,  $\sigma(T)$ and $\sigma_{pt}(T)$ denote the spectrum and the collection of eigenvalues of $T$,  respectively.  Since biholomorphic automrphisms turn out to be isometries for the  Carath\'eodory metric, cf. \cite[Proposition IV.1.6]{FV} \textit{\textbf{throughout this article, the word isometry would either mean a biholomorphic automorphism of $\mathcal{B}$ or its linear representation.}} As Harris \cite{HA} and Franzoni \cite{FR} give the general form of a holomorphic isometry acting on $\mathcal{B}$ and discuss its linear representation, we begin with a quick review of their work which forms the basis of our study.\\
\begin{remark}\label{b}
 \cite[Chapter 1, (3.4)]{NA}  For a contraction $T \in B(H)$, an application of continuous functional calculus tells that 
  \[T(I-T^*T)^{1/2}=(I-TT^*)^{1/2}T.\]
   \end{remark}
For $B \in \mathcal{B}$, the map $T_B:\mathcal{B} \rightarrow B(K,H)$ defined by 
\begin{eqnarray}\label{eq:14}
\hspace{10mm} T_B(A)=(I-BB^*)^{-1/2}(A+B)(I+B^*A)^{-1}(I-B^*B)^{1/2}
\end{eqnarray}
is a holomorphic automorphism of $\mathcal{B}$ such that $T_B(0)=B$ (see Remark \ref{b}) and ${T_B}^{-1}=T_{-B}$. $T_B$ is called the M\"{o}bius transformation \cite[Theorem 2]{HA}. 

\begin{theorem}\cite[Theorem 1.6]{FR}\label{t1}
 Every biholomorphic mapping $h:\mathcal{B} \rightarrow \mathcal{B}$ such that $h(0)=B$ is of the form 
\[h=T_B \circ L,\,\,\,\,B \in \mathcal{B}\]
where $L$ is a norm preserving surjective linear map of $B(K,H)$ and $T_B$ is a M\"{o}bius transformation (\ref{eq:14}) defined on $\mathcal{B}$.
 \end{theorem}
So far there has been no restriction on the dimensions of $K$ and $H$. \\
\textbf{\textit{From now onwards we will assume \pmb{$H$} and \pmb{$K$} have different dimensions}}, i.e. they have orthonormal bases of different cardinalities.\\ 
 
Let $K$ and $H$ be Hilbert spaces having different dimensions. Then for a norm preserving surjective linear isometry $L$  defined on $B(K,H)$, there exist $U \in U(H)$ and $V \in U(K)$ such that 
 \begin{eqnarray}\label{eq:12}
L(A)=UAV,\,\,\,\,\,\,\,\forall A \in B(K,H),
\end{eqnarray}
 cf. \cite[Theorem 4.6]{FR}. $L$ is the generalization of a unitary operator to Banach spaces. Collection of all the biholomorphic mappings of the form $h$ defined in Theorem \ref{t1} forms a group which will be denoted by $Aut(\mathcal{B})$. This group acts transitively on $\mathcal{B}$, see \cite[Corollary 2]{HA}.\\
For $A, B \in \mathcal{B}$, the Carath\'eodory distance $C_D$ \cite[Sec. 3]{FR} is defined as 
\[C_D(A,B)=\rho\bigg(0, \|T_{-B}(A)\|\bigg)=tanh^{-1}\bigg(\|T_{-B}(A)\|\bigg)\]
where $\rho$ is the Poincar\'e metric on the open unit ball in $\mathbb{C}$ and $T_{-B}$ is the inverse of M\"{o}bius transformation $T_B$ defined in (\ref{eq:14}). For definition and details on Carath\'eodory metric, refer \cite{FV}.\\
A linear representation of $Aut(\mathcal{B})$ due to Franzoni \cite[section 5]{FR}  is as follows.

Consider the Hermitian form $\mathcal{S}:(H \oplus K) \times (H \oplus K) \rightarrow \mathbb{C}$ defined as
\[\mathcal{S}\big((h_1,k_1),(h_2,k_2)\big)=\left<h_1,h_2\right>-\left<k_1,k_2\right>.\]
Let $\mathcal{G}$ be the group of all bijective linear transformations defined on $H \oplus K$ leaving $\mathcal{S}$ invariant. The general form of an element of $\mathcal{G}$ is $T=\left[ {\begin{array}{cc}
   B & C \\
   D & E \\
  \end{array} } \right]$ where $B:H \rightarrow H$, $C:K \rightarrow H$, $D:H \rightarrow K$ and $E:K \rightarrow K$ are bounded linear operators satisfying
  \begin{equation}\label{eq:1}
  \begin{aligned}
  B^*B-D^*D&=I,\\
  E^*E-C^*C&=I,\\
   C^*B-E^*D&=0,\\
  BB^*-CC^*&=I,\\
  EE^*-DD^*&=I,\,\,\text{and}\\
  DB^*-EC^*&=0.
  \end{aligned}
  \end{equation}
  The above equations imply that $B$
and $E$ are invertible operators, see \cite[Lemma 5.1]{FR}.
\begin{remark}
For $T=\left[ {\begin{array}{cc}
   B & C \\
   D & E \\
  \end{array} } \right] \in \mathcal{G}$,  $T^{-1}=\left[ {\begin{array}{cc}
   B^* & -D^* \\
   -C^* & E^* \\
  \end{array} } \right]$, see \cite[page 63]{FR}.
\end{remark}
 The following theorem shows the identification of $Aut(\mathcal{B})$ with $\mathcal{G}$. It has been proved in \cite{FR} but here we give the detailed version of the proof.
  \begin{theorem}\label{c} \cite[Theorem 5.3]{FR}
Let  $H$ and $K$ be Hilbert spaces of different dimensions. The map $\psi:\mathcal{G} \rightarrow Aut(\mathcal{B})$ defined as $T \mapsto \psi(T)$, where $T=\left[ {\begin{array}{cc}
   B & C \\
   D & E \\
  \end{array} } \right]$  and $\big(\psi(T)\big)(X)=(BX+C)(DX+E)^{-1}$, $X \in \mathcal{B}$, is a surjective homomorphism.
  \end{theorem}
  \begin{proof}
  Well-definedness and homomorphicity of $\psi$ are self evident. Let $h_A=T_A \circ L \in Aut(\mathcal{B})$, $A \in \mathcal{B}$. As $\psi$ is a  homomorphism, it is sufficient to provide pre-images of $T_A$ and $L$ separately under $\psi$. Notice that  for $L \in Aut(\mathcal{B})$ of the form $L(X)=UXV$ (cf. \ref{eq:12}), $T_1=\left[ {\begin{array}{cc}
   U & 0 \\
   0 & V^{-1} \\ 
  \end{array} } \right]$ does the job. Next, we need to find a suitable element $T_2=\left[ {\begin{array}{cc}
   B & C \\
   D & E \\
  \end{array} } \right] \in \mathcal{G}$ satisfying
  \[(BX+C)(DX+E)^{-1}=(I-AA^*)^{-1/2}(X+A)(I+A^*X)^{-1}(I-A^*A)^{1/2}\]
  for all $X \in \mathcal{B}$. Putting $X\equiv 0$ in the above equation gives $CE^{-1}=(I-AA^*)^{-1/2}A(I-A^*A)^{1/2}=A$  by Remark \ref{b}, thus $C=AE$. Also by (\ref{eq:1}),  $E^*E-I=C^*C=E^*A^*AE$, i.e. $E^*E=E^*A^*AE+I$ implying  $E^*=E^*A^*A+E^{-1}$ and $EE^*=EE^*A^*A+I$. Hence $EE^*(I-A^*A)=I$ or $EE^*=(I-A^*A)^{-1}$. Choose \(E\) to be the positive square root of $(I-A^*A)^{-1}$. This gives  $C=A(I-A^*A)^{-1/2}$.
  Next the relation $BB^*=I+CC^*$ and Neumann series expansion for the operator $(I-A^*A)^{-1}$ yield $BB^*=(I-AA^*)^{-1}$. Choose \(B\) to be the positive square root of the operator $(I-AA^*)^{-1}$. Lastly $DB^*=EC^*$ gives  $D=(I-A^*A)^{-1/2}A^*$. Observe that 
  \begin{eqnarray}\label{eq:20}
  D=C^*. 
  \end{eqnarray}
  So using Remark \ref{b}, $T_2=\left[ {\begin{array}{cc}
   (I-AA^*)^{-1/2} & {(I-AA^*)^{-1/2}}A \\
   (I-A^*A)^{-1/2}A^* & (I-A^*A)^{-1/2} \\
  \end{array} } \right]$.
  \end{proof}
  \begin{remark}\label{d}
  The elements in the centre $Z(\mathcal{G})$ of $\mathcal{G}$ are all the unit constant multiples of identity \cite[page 65]{FR} and $ker(\psi)=Z(\mathcal{G})$,  thereby making $\widetilde{\psi}:\mathcal{G}/Z(\mathcal{G}) \rightarrow Aut(\mathcal{B})$ a surjective isomorphism.
  \end{remark}
\section{Normal isometries of $\mathcal{B}$}
  We begin by presenting the general form of an element of $\mathcal{G}$ which is a direct application of the  proof of Theorem \ref{c} and Remark \ref{d}.
  \begin{proposition} \label{p}
  A general element of $\mathcal{G}$ is represented as
  \[T=e^{i \theta}\left[ {\begin{array}{cc}
   (I-AA^*)^{-1/2}U & {(I-AA^*)^{-1/2}}AV \\
   (I-A^*A)^{-1/2}A^*U & (I-A^*A)^{-1/2}V \\
  \end{array} } \right]\]
  where $ \theta \in \mathbb{R}$, $A \in \mathcal{B}$, $U \in U(H)$ and $V \in U(K)$.
 \end{proposition}
 Observe that a general element $T$ of $\mathcal{G}$  in Proposition \ref{p} is the composition of a self-adjoint element $\left[ {\begin{array}{cc}
   (I-AA^*)^{-1/2} & {(I-AA^*)^{-1/2}}A \\
   (I-A^*A)^{-1/2}A^* & (I-A^*A)^{-1/2} \\
  \end{array} } \right]$  and a unitary operator $e^{i \theta} \left[ {\begin{array}{cc}
   U & 0 \\
   0 & V \\ 
  \end{array} } \right]$.\\
 
 For the sake of simplicity, let $B=(I-AA^*)^{-1/2}$, $C={(I-AA^*)^{-1/2}}A$ and $E=(I-A^*A)^{-1/2}$. So a general element of $\mathcal{G}$ would be 
  $T=\left[ {\begin{array}{cc}
   BU & CV \\
   C^*U & EV \\
  \end{array} } \right]$.  Observe that the fact $B$ and $E$ are positive operators simplifies (\ref{eq:1}) to 
  \begin{equation}\label{eq:19}
  \begin{aligned}
  B^2-CC^*&=I,\\
  E^2-C^*C&=I,\,\,\text{and}\\
   BC&=CE.
  \end{aligned}
  \end{equation}
 
We now discuss the characterising conditions for an isometry in $\mathcal{G}$ to be normal.\\\\

 \textbf{Proof of Proposition \ref{P1}.} We will compute the expressions of $TT^*$ and $T^*T.\\$ $TT^*=\left[ {\begin{array}{cc}
   BU & CV \\
   C^*U & EV \\
  \end{array} } \right]\left[ {\begin{array}{cc}
   U^{-1}B & U^{-1}C \\
   V^{-1}C^* & V^{-1}E \\
  \end{array} } \right]=\left[ {\begin{array}{cc}
   B^2+CC^* & BC+CE \\
   C^*B+EC^* & C^*C+E^2 \\
  \end{array} } \right]$\\\\ 
  and $T^*T=\left[ {\begin{array}{cc}
   U^{-1}B & U^{-1}C \\
   V^{-1}C^* & V^{-1}E \\
  \end{array} } \right]\left[ {\begin{array}{cc}
   BU & CV \\
   C^*U & EV \\
  \end{array} } \right]$\\\\
   $=\left[ {\begin{array}{cc}
   U^{-1}(B^2+CC^*)U & U^{-1}(BC+CE)V \\
   V^{-1}(C^*B+EC^*)U & V^{-1}(C^*C+E^2)V \\
  \end{array} } \right]$. 
 Equating the respective entries of  $TT^*$ and $T^*T$, we get
  \begin{align*}
      U(B^2+CC^*)&=(B^2+CC^*)U,\\
      U(BC+CE)&=(BC+CE)V,\\
      V(C^*B+EC^*)&=(C^*B+EC^*)U, \,\text{and}\\
      V(C^*C+E^2)&=(C^*C+E^2)V
  \end{align*}
  which in view of  (\ref{eq:19}) simplifies  to $UB^2=B^2U,\, UBC=BCV,\, \text{and}\,\, VE^2=E^2V.$ Hence $UB=BU$ and $VE=EV$. Also $UBC=BUC=BCV$ implies $UC=CV$. Thus $T$ is normal if and only if
      $UB=BU$,
     $ UC=CV$ and
     $ VE=EV.$

 As a corollary, we present the form of self-adjoint isometries.
 \begin{corollary}\label{c1}
 Let $T=
 \left[ {\begin{array}{cc}
   BU & CV \\
   C^*U & EV \\
  \end{array} } \right] \in \mathcal{G}$ be a normal isometry. Then $T$ is self-adjoint if and only if both $U$ and $V$ are self-adjoint operators.
 \end{corollary}
 \begin{proof}
 $T$ is self-adjoint if and only if $T=T^*$, i.e.
 \[\left[ {\begin{array}{cc}
   BU & CV \\
   C^*U & EV \\
  \end{array} } \right]=\left[ {\begin{array}{cc}
   U^{-1}B & U^{-1}C \\
   V^{-1}C^* & V^{-1}E \\
  \end{array} } \right]\] which on equating the respective entries gives $BU=U^{-1}B,\,CV=U^{-1}C\,$ and $EV=V^{-1}E$. Now corollary follows using Proposition \ref{P1} and the fact that a unitary operator is self-adjoint if and only if it is an involution, i.e. order two element.
 \end{proof}
 
 \begin{proposition}\label{P3}
  Let $T=\left[ {\begin{array}{cc}
   BU & CV \\
   C^*U & EV \\
  \end{array} } \right] \in \mathcal{G}$, then $T$ is unitary if and only if $C \equiv 0$.
 \end{proposition}
 \begin{proof}
 For $T=\left[ {\begin{array}{cc}
   BU & CV \\
   C^*U & EV \\
  \end{array} } \right] \in \mathcal{G}$, $T$ is unitary if and only if $T^*=T^{-1}$ , i.e. $\left[ {\begin{array}{cc}
   U^{-1}B & U^{-1}C \\
   V^{-1}C^* & V^{-1}E \\
  \end{array} } \right]=\left[ {\begin{array}{cc}
   U^{-1}B & -U^{-1}C \\
   -V^{-1}C^* & V^{-1}E \\
  \end{array} } \right]$. It is easy to see that respective entries are equal if and only if  $C \equiv 0$.  
  \end{proof}
  \begin{corollary}
   A unitary operator in $\mathcal{G}$ is of the form  $\left[ {\begin{array}{cc}
   W & 0 \\
   0 & W' \\
  \end{array} } \right]$ for some $W \in U(H)$ and $W' \in U(K)$.
  \end{corollary}
  \begin{proof}
    Let $T=\left[{\begin{array}{cc}
   BU & CV \\
   C^*U & EV \\
  \end{array}}\right] \in \mathcal{G}$ be a unitary operator. Then $T=\left[ {\begin{array}{cc}
   BU & 0 \\
   0 & EV \\
  \end{array} } \right]$ where   $B$ and $E$ are involution operators by (\ref{eq:19}). Hence  $B$ and $E$ are unitary as these are  self-adjoint also. So $T=\left[ {\begin{array}{cc}
   W & 0 \\
   0 & W' \\
  \end{array} } \right]$ where $W=BU$ and $W'=EV$.
  \end{proof}

\section{Normal isometries for \pmb{$\mathcal{B} \subseteq B(\mathbb{C}^n,H)$}} 
 
 \textbf{\textit{From now onwards, we will assume}  \pmb{$K=\mathbb{C}^n$}} \textbf{\textit{and a normal isometry would mean  non-unitary normal isometry}}.
  
  \medskip
 Let us understand the general form of elements of $\mathcal{G}$  more explicitly for finite dimensional \(K\).
 
From the proof of Theorem \ref{c}, $E=(I-A^*A)^{-1/2}$ is a positive invertible operator where $A \in \mathcal{B}$. Since $\|A^*A\|<1,\, \sigma(A^*A) \in [0,1)$. This implies $\sigma(I-A^*A) \in (0,1]$ and hence  $\sigma(E) \in [1, \infty)$.  Thus for $T=\left[ {\begin{array}{cc}
   BU & CV \\
   C^*U & EV \\
  \end{array} } \right] \in \mathcal{G}$,  \(E\) and \(B\) are positive invertible operators with \(E\) having eigenvalues $\geq$ 1. Let $E$ decompose the space $\mathbb{C}^n$ orthogonally into eigenspaces as $\mathbb{C}^n=\overset{p}{\underset{i=1}{\oplus}}K_i$ such that $E\restriction_{K_i}=a_i\,I$, $a_i \geq 1$.\\
  Let $dim\,(K_i)=k_i$ and let $\beta_i=\{e_{{i}_{j}},\,\,j=1,\,2,\,...,\,k_i\}$ be an orthonormal basis of $K_i$ for each $i=1,\,2,\,...,\,p$.\\
  Let $C(e_{i_j})=\xi_{i_j}$. This implies $C^*(h)=\sum\limits_{r=1}^{p}\sum \limits _{s=1}^{k_r}{\left<h,\xi_{r_s}\right>\,e_{r_s}}$. \\
  The relation $E^2-C^*C=I$ in (\ref{eq:19}) yields 
   \[({a_i}^2-1)\,e_{i_j}=\sum\limits_{r=1}^{p}\sum \limits _{s=1}^{k_r}{\left<\xi_{i_j},\xi_{r_s}\right>\,e_{r_s}}.\]
  Taking inner product in the above equation with $e_{i_j}$ gives ${a_i}^2=1+\|\xi_{i_j}\|^2$ which further implies that for all $j,\,j'\in\{1,\,2,\,...,\,k_i\}$, $\|\xi_{i_j}\|=\|\xi_{i_{j'}}\|=\delta_i$ (say). Hence ${a_i}^2=1+{\delta_i}^2$ for all $i=1,\,2,\,...,\,p$.\\
  Now take $r_s \neq i_j$ and take inner product in the same equation with $e_{r_s}$, we get $\left<\xi_{i_j},\xi_{r_s}\right>=0$. This gives $\xi_{i_j}=\xi_{r_s}$, $i_j \neq r_s$ only if $\xi_{i_j}=0$, i.e. non-zero  ${\xi_{i_j}}$'s are distinct and are mutually orthogonal. Also some of the  ${\xi_{i_j}}$'s may be zero as well  and if $\xi_{i_j}=0$ then $\xi_{i_{j'}}=0$ for all $j'\in\{1,\,2,\,...,\,k_i\}$. This will make the corresponding $a_i=1$. \\
  Let us re-define the orthogonal decomposition of $\mathbb{C}^n$ via $E$ as follows.\\
  Out of $p$ number of eigenspaces of $E$ let the first $l$ number of eigenspaces ($l \leq p$) be such that $C$ restricted to these is throughout non-zero and it vanishes on the orthogonal complement. This makes the eigensvalues of $E$ strictly bigger than 1 for these $l$ eigenspaces and $E=I$ on the orthogonal complement.  \\
  Let $K'= \overset{l}{\underset{i=1}{\oplus}}{K_i}$. So, $\mathbb{C}^n=K' \oplus {K'}^{\perp}$ where $E\restriction_{K_i}=a_i\,I$, $a_i > 1$, $i=1,\,2,\,...,\,l$ and  $E\restriction_{{K'}^{\perp}}=I.$ Let $dim\,K'=k$ which gives $\sum\limits_{i=1}^{l}k_i=k$.\\
  Hence $C\restriction_{{K'}^{\perp}}\equiv 0$ and due to mutual orthogonality of ${\xi_{i_j}}$'s, $C\restriction_{K'}$ becomes bijective  where $C(e_{i_j})=\xi_{i_j}$,  $i=1,\,2,\,...,\,l$ and $j =1,\,2,\,...,\,k_i$. Hence \\
  $C^*(\xi_{i_j})={\delta_i}^2\,e_{i_j},\,\,j=1,\,2,\,...,\,k_i,\,\,i=1,\,2,\,...,\,l$ and
       $C^* \equiv 0$ on  $(ran\,C)^{\perp}$ where $ran\,C$ denotes the range of $C$.\\
       At last, the relation $B^2-CC^*=I$ yields 
  $B(\xi_{i_j})=a_i\,\xi_{i_j},\,\,j=1,\,2,\,...,\,k_i,\,\,i=1,\,2,\,...,\,l$ and 
       $B = I$ on $(ran\,C)^{\perp}.$\\
       Let $M_i=span\,\{\xi_{i_j},\,\,j=1,\,2,\,...,\,k_i\}$ for each $i=1,\,2,\,...,\,l$. So  $\overset{l}{\underset{i=1}{\oplus}}{M_i}=ran\,C$  and $B$ decomposes $H$ orthogonally into eigenspaces as $H=ran\,C \oplus {(ran\,C)}^{\perp}$ where $B\restriction_{M_i}=a_i\,I$. Also $dim\,M_i=dim\,K_i=k_i$. Let ${\beta_i}'=\{\xi_{i_j},\,\,j=1,\,2,\,...,\,k_i\} $ which is  an orthogonal basis of $M_i$ for each $i=1,\,2,\,...,\,l$.\\
       
   \textit{\textbf{In the rest of the article we reserve the notations \pmb{$I$} and \pmb{$J_i$} for the sets \pmb{$\{1,\,2,\,...,\,l\}$} and \pmb{$\{1,\,2,\,...,\,k_i\}$}, respectively where \pmb{$i \in I$}}}.\\
  As a summary of the preceding analysis, we have the following information. 
  \begin{remark}\label{r6}
Let $T=\left[ {\begin{array}{cc}
   BU & CV \\
   C^*U & EV \\
  \end{array} } \right] \in \mathcal{G}$. Then  $E$ decomposes $\mathbb{C}^n$ orthogonally into eigenspaces as $\mathbb{C}^n=\overset{l}{\underset{i=1}{\oplus}}{K_i} \oplus {K'}^{\perp}$ where 
  \begin{eqnarray}\label{eq:33}
  K'= \overset{l}{\underset{i=1}{\oplus}}{K_i},\,\,\text{and}
  \end{eqnarray}
  \begin{equation}\label{eq:28}
      \begin{aligned}
       E\restriction_{K_i}&=a_i \,I, \,\, a_i > 1,\\
       E\restriction_{{K'}^{\perp}}&=I.
    \end{aligned}
  \end{equation}
  Let $dim\,(K')=k$ and   $dim\,(K_i)=k_i$ for each $i \in I$,  hence
  \begin{eqnarray*}
  \sum\limits_{i=1}^{l}{k_i}=k.
  \end{eqnarray*}
  \begin{eqnarray}\label{eq:24}
  \beta_i=\{e_{i_j},\,\,j \in J_i\} 
  \end{eqnarray}
  is an orthonormal basis of $K_i$ for each $i \in I$.
  \begin{equation}\label{eq:29}
      \begin{aligned}
       C(e_{i_j})&=\xi_{i_j} (\neq 0),\,\,j \in J_i,\,\,i \in I,\\
       C & \equiv 0\,\,\text{on}\,\,{K'}^{\perp}.
    \end{aligned}
  \end{equation}
 As  ${\xi_{i_j}}$'s are distinct and  mutually orthogonal,  $C\restriction_{K'}$ is a  bijective operator and $dim\,(ran\,C)=k=dim\,K'$.\\
  Also $\|\xi_{i_j}\|=\|\xi_{i_{j'}}\|$ for every $j,\,j' \in K_i$. Let
  \begin{eqnarray}\label{eq:26}
  \|\xi_{i_j}\|=\delta_i,\,\, j \in J_i.
  \end{eqnarray}
  Hence 
  \begin{eqnarray}\label{eq:39}
  {a_i}^2=1+{\delta_i}^2.
  \end{eqnarray}
   \begin{equation}\label{eq:30}
      \begin{aligned}
       C^*(\xi_{i_j})&={\delta_i}^2e_{i_j},\,\,j \in J_i,\,\,i \in I,\\
       C^* &\equiv 0\,\,\text{on}\,\,  (ran\,C)^{\perp}.
    \end{aligned}
  \end{equation}
   Let $M_i=span\,\{\xi_{i_j},\,\,j \in J_i\}$ for each $i \in I$. Then  
   \begin{eqnarray}\label{eq:32}
  ran\,C=\overset{l}{\underset{i=1}{\oplus}}M_i
  \end{eqnarray}
  where $dim\,(M_i)=k_i=dim\,(K_i)$. 
  \begin{eqnarray}\label{eq:27}
  {\beta_i}'=\{\xi_{i_j},\,\,j \in J_i\} 
  \end{eqnarray}
  is an orthogonal basis of $M_i$ for each $i \in I$. $B$
decomposes $H$ orthogonally as $H=\overset{l}{\underset{i=1}{\oplus}}M_i \oplus (ran\,C)^{\perp}$ where 
\begin{equation}\label{eq:31}
      \begin{aligned}
       B\restriction_{M_i}&=a_i\,I,\,\,j \in J_i,\,\,i\in I,\\
       B &= I\,\,\text{on}\,\, (ran\,C)^{\perp}.
    \end{aligned}
  \end{equation}
  \end{remark}
 
 Now we are ready to characterize non-unitary normal isometries for $K=\mathbb{C}^n$. \\\\
  
 \textbf{Proof of Proposition \ref{h}.} 
  Let $T$ be a non-unitary normal isometry. Then in view of Proposition \ref{P1},  $UB=BU$  and  $VE=EV$ which implies  $U(M_i)=M_i$ and $V(K_i)=K_i$ using (\ref{eq:31}) and (\ref{eq:28}) respectively. Also $UC(e_{i_j})=U(\xi_{i_j})$ by (\ref{eq:29}). As $U(M_i)=M_i$, let $U(\xi_{i_j})=\sum\limits_{p=1}^{k_i}{{\alpha_{j_p}}^{(i)}}{\xi_{i_p}}$ using (\ref{eq:27}) where ${{\alpha_{j_p}}^{(i)}}\in \mathbb{C}$. Similarly, as $V(K_i)=K_i$, let  $V(e_{i_j})=\sum\limits_{p=1}^{k_i}{{\gamma_{j_p}}^{(i)}}{e_{i_p}}$ using (\ref{eq:24}) where ${{\gamma_{j_p}}^{(i)}}\in \mathbb{C}$. This gives $CV(e_{i_j})=\sum\limits_{p=1}^{k_i}{{\gamma_{j_p}}^{(i)}}{\xi_{i_p}}$. As $UC=CV$, $UC(e_{i_{j}})=CV(e_{i_{j}})$ which implies ${{\alpha_{j_p}}^{(i)}}={{\gamma_{j_p}}^{(i)}}$ for all $i \in I,\,j,\,p\in J_i$ using mutual orthogonality of distinct ${\xi_{i_{p}}}$'s. Hence
  \begin{eqnarray}\label{eq.}
  [U(\xi_{i_j})]_{{\beta_i}'}=[V(e_{i_j})]_{{\beta_i}}
  \end{eqnarray}
  and   ${[U\restriction_{M_i}]}_{{\beta_i}'}={[V\restriction_{K_i}]}_{{\beta_i}}$ for every $i \in I$.\\
  Conversely, tracing back the steps of forward part together with the facts that $B=I$ on $(ran\,C)^{\perp}$, $E=I$ on ${K'}^{\perp}$ and $C\equiv 0$ on ${K'}^{\perp}$ yield  $UB=BU$, $VE=EV$ and $UC=CV$. 
  
 \begin{corollary}\label{e}
  If $T=\left[ {\begin{array}{cc}
   BU & CV \\
   C^*U & EV \\
  \end{array} } \right] \in \mathcal{G}$ is a non-unitary normal isometry then $T=\overset{l}{\underset{i=1}{\oplus}}{T_i} \oplus T' \oplus T''$ where $T_i=T \restriction_{M_i \oplus K_i}$, $i \in I$, $T'=U\restriction_{{(ran\,C)}^{\perp}}$ and $T''=V\restriction_{{K'}^{\perp}}$.
  \end{corollary}
  \begin{proof}
  Note that  ${\beta_i}' \cup \beta_i=\{(\xi_{i_j},0),\,\,\,\,j\in J_i\}\cup \{(0,e_{i_j}),\,\,\,\,j\in J_i\}$ is a basis of $M_i \oplus K_i$ which follows from (\ref{eq:27}) and (\ref{eq:24}). Now using (\ref{eq:31}), (\ref{eq:30}) and (\ref{eq.}), we get \\
  
 $ \left[ {\begin{array}{cc}
   BU & CV \\
   C^*U & EV \\
  \end{array} } \right]\left[ {\begin{array}{c}
   \xi_{i_j} \\
  0 \\
  \end{array} } \right]=\left[ {\begin{array}{c}
   BU(\xi_{i_j}) \\
  C^*U(\xi_{i_j}) \\
  \end{array} } \right]=\left[ {\begin{array}{c}
   a_iU(\xi_{i_j}) \\
  {\delta_i}^2V(e_{i_j}) \\
  \end{array} } \right]$ $\in M_i \oplus K_i$.\\
  
 Similarly, the fact that $UC=CV$, (\ref{eq:29}) and (\ref{eq:28}) give\\
 
 $\left[ {\begin{array}{cc}
   BU & CV \\
   C^*U & EV \\
  \end{array} } \right]\left[ {\begin{array}{c}
   0 \\
  e_{i_j} \\
  \end{array} } \right]=\left[ {\begin{array}{c}
   CV(e_{i_j}) \\
  EV(e_{i_j}) \\
  \end{array} } \right]=\left[ {\begin{array}{c}
   U(\xi_{i_j}) \\
  a_iV(e_{i_j}) \\
  \end{array} } \right] \in M_i \oplus K_i$. \\
  
  This implies $T(M_i \oplus K_i)=M_i \oplus K_i$ as $T$ is one-one.  For $h \in (ran\,C)^{\perp}$, the fact that $ran\,C$ reduces $U$ together with (\ref{eq:31}) and (\ref{eq:30}) yield\\ 
  
  $\left[ {\begin{array}{cc}
   BU & CV \\
   C^*U & EV \\
  \end{array} } \right]\left[ {\begin{array}{c}
   h \\
  0 \\
  \end{array} } \right]=\left[ {\begin{array}{c}
   U(h) \\
  0 \\
  \end{array} } \right] \in (ran\,C)^{\perp}$\\
  
  implying  $T((ran\,C)^{\perp})=(ran\,C)^{\perp}$. Let $T'=T\restriction_{(ran\,C)^{\perp}}=U\restriction_{(ran\,C)^{\perp}}$. \\
  Also for $\bs{\mr{z}} \in {K'}^{\perp}$, the fact that $K'$ reduces $V$ together with (\ref{eq:29}) and (\ref{eq:28}) imply
  $\left[ {\begin{array}{cc}
   BU & CV \\
   C^*U & EV \\
  \end{array} } \right]\left[ {\begin{array}{c}
   0 \\
   \bs{\mr{z}}\\
  \end{array} } \right]=\left[ {\begin{array}{c}
   0 \\
  V(\bs{\mr{z}}) \\
  \end{array} } \right] \in {K'}^{\perp}$
  giving $T({K'}^{\perp})={K'}^{\perp}$. Let $T''=T\restriction_{{K'}^{\perp}}=V\restriction_{{K'}^{\perp}}$.
  Hence we have $T=\overset{l}{\underset{i=1}{\oplus}}{T_i} \oplus T' \oplus T''$.
  \end{proof}
  \begin{corollary}\label{c2}
   $[T_i]_{{\beta_i}' \cup {\beta_i}}=\left[ {\begin{array}{cc}
   a_iR_i & R_i \\
   {\delta_i}^2 R_i & a_i R_i \\
  \end{array} } \right]_{2k_i \times 2k_i}$ where $R_i={[U\restriction_{M_i}]}_{{\beta_i}'}={[V\restriction_{K_i}]}_{{\beta_i}}$.
\end{corollary}
\begin{proof}
Using the expressions of $T(\xi_{i_{j}},0)$ and $T(0,e_{i_{j}})$ from the proof of Corollary \ref{e} and the expressions of $U(\xi_{i_{j}})$ and $V(e_{i_{j}})$ from the proof of Proposition \ref{h} along with (\ref{eq.}), we get 
\[[T_i]_{{\beta_i}' \cup {\beta_i}}=[{t_{r_s}}^{(i)}]_{2k_i \times 2k_i}\]
where $r$ is the number of rows and $s$ is the number of columns and
\begin{align*}
{t_{r_s}}^{(i)} =
\left\{
	\begin{array}{llll}
		\left<T(\xi_{i_{s}},0),(\xi_{i_{r}},0)\right>  & \mbox{if } 1 \leq r,s \leq k_i, \\
		\left<T(\xi_{i_{s}},0),(0,e_{i_{p}})\right> & \mbox{if } r=k_i+p,\, 1 \leq p \leq k_i\,\, \mbox{and}\,\, 1 \leq s \leq k_i, \\
		\left<T(0,e_{i_{p}}),(\xi_{i_{r}},0)\right> & \mbox{if }  1 \leq r \leq k_i \,\, \mbox{and}\,\,s=k_i+p,\, 1 \leq p \leq k_i, \\
		\left<T(0,e_{i_{p}}),(0,e_{i_{m}})\right> & \mbox{if }  r=k_i+m \,\, \mbox{and}\,\,s=k_i+p,\, 1 \leq m,p \leq k_i.
	\end{array}
\right.
    \end{align*}
    The above analysis in compact form gives \\
    $[T_i]_{{\beta_i}' \cup {\beta_i}}=\left[ {\begin{array}{cc}
   a_iR_i & R_i \\
   {\delta_i}^2 R_i & a_i R_i \\
  \end{array} } \right]_{2k_i \times 2k_i}$ where $R_i={[U\restriction_{M_i}]}_{{\beta_i}'}={[V\restriction_{K_i}]}_{{\beta_i}}$.

\end{proof}
  Next we describe the unitary elements in $\mathcal{G}$ for $K=\mathbb{C}^n$.\\\\
  
 \textbf{Proof of Proposition \ref{z}.}
Let $T=\left[ {\begin{array}{cc}
   BU & CV \\
   C^*U & EV \\
  \end{array} } \right] \in \mathcal{G}$ be a unitary operator. By Proposition \ref{P3},  $C\equiv 0$. This would make $(ran\,C)^{\perp}=H$ and hence $B=I$ on $H$ by (\ref{eq:31}). Also as $\xi_{i_j}=0$ for all $i\in I$ and $j \in J_i$, we have $a_i=1$ for every $i \in I$ using (\ref{eq:26}) and (\ref{eq:39}), hence    $E=I$ by  (\ref{eq:28}). Thus $T=\left[ {\begin{array}{cc}
   U & 0 \\
   0 & V \\
  \end{array} } \right]$.
 
  \section{Fixed points of non-unitary normal isometries for \pmb{$\mathcal{B} \subseteq B(\mathbb{C}^n,H)$}}
 Due to a result of Kaup \cite{KA}, every holomorphic automorphism on $\mathcal{B}$ is the restriction of a holomorphic map defined in some neighbourhood of $\overline{\mathcal{B}}$ taking values in $B(K,H)$. What follows is a  result from \cite{FR} which guarantees existence of fixed points in $\overline{\mathcal{B}}$. 
  \begin{theorem}\cite[Corollary 2.4]{FR}  Every holomorphic automorphism of the open unit ball $\mathcal{B} $ in $B(\mathbb{C}^n,H)$ has a fixed point in $\overline{\mathcal{B}}$.
  \end{theorem}
  Also, to have a representation of fixed points for linear isometries, we need to ensure that the expression $(BA+C){(DA+E)}^{-1}$  which appears in Theorem \ref{c} is meaningful for $A \in \overline{\mathcal{B}}$. As (\ref{eq:20}) says $D=C^*$ , we will show that
  \begin{lemma}
  $C^*A+E$ is invertible for $A \in \overline{\mathcal{B}}$.
  \end{lemma}
  \begin{proof}
  As \(E\) is a positive invertible operator,  for $\textbf{z}\,(\neq 0) \in \mathbb{C}^n$, (\ref{eq:19}) yields\\
  $\left<E^{-1}(E^2-C^*C){E^{-1}}(\textbf{z}),\tb{z}\right>=\left<E^{-1}{E^{-1}}(\tb{z}),\tb{z}\right> >0$. In particular \linebreak
  $\left<E^{-1}C^*C{E^{-1}}(\tb{z}),\tb{z}\right> <1$ for $\|\tb{z}\|=1$. Since the unit sphere in $\mathbb{C}^n$ is a compact set, we have 
  $\underset{\|\tb{z}\|=1}
  {sup}\left<E^{-1}C^*{(E^{-1}C^*)}^*(\tb{z}),\tb{z}\right> <1$. This means $\Big\|E^{-1}C^*\Big\|^2 <1$ and $\Big\|E^{-1}C^*A\Big\| <1$ for $A \in \overline{\mathcal{B}}$ which makes $I+E^{-1}C^*A$ and hence $E+C^*A$ an invertible operator. 
  \end{proof}
   Let us now have a representation of  fixed points for linear isometries.
   
  \begin{proposition}\label{P}
  Let $h \in Aut(\mathcal{B})$ be defined as $$h(A)=(BUA+CV){(C^*UA+EV)}^{-1}.$$  Let $T_h=\left[ {\begin{array}{cc}
   BU & CV \\
   C^*U & EV \\
  \end{array} } \right]$ be a linear isometry that corresponds to $h$, cf. Theorem \ref{c}. Let $\{\boldsymbol{\mathrm{z}_r},\,\,r=1,\,2,\,...,\,n\}$ be a basis of $\mathbb{C}^n$ and $F \in B({\mathbb{C}}^n,H)$ be such that
  $\left[{\begin{array}{c}
  F(\boldsymbol{\mr{z}_r})\\
  \bs{\mr{z}_r}\\
  \end{array}}\right]$ are eigenvectors of $T_h$, $r=1,\,2,\,...,\,n$, then $F$ is a fixed point of $h$. Conversely, if $F$ is a fixed point of \(h\) and there exists an eigenbasis  $\{\boldsymbol{\mr{z}_r},\,\,r=1,\,2,\,...,\,n\}$ of $C^*UF+EV$, then \\  $\left[{\begin{array}{c}
  F(\boldsymbol{\mr{z}_r})\\
  \boldsymbol{\mr{z}_r}\\
  \end{array}}\right]$ are eigenvectors of $T_h$.
  \end{proposition}
  \begin{proof}
  Let $\left[{\begin{array}{c}
  F(\boldsymbol{\mr{z}_r})\\
  \bs{\mr{z}_r}\\
  \end{array}}\right]$ be eigenvectors of $T_h$ having eigenvalues $\lambda_r$ (say),  i.e.  $\left[ {\begin{array}{cc}
   BU & CV \\
   C^*U & EV \\
  \end{array} } \right] \left[{\begin{array}{c}
  F(\bs{\mr{z}_r})\\
  \bs{\mr{z}_r}\\
  \end{array}}\right]=\lambda_r\left[{\begin{array}{c}
  F(\bs{\mr{z}_r})\\
  \bs{\mr{z}_r}\\
  \end{array}}\right]$. This gives\\
  $BU(F(\bs{\mr{z}_r}))+CV(\bs{\mr{z}_r})=F(C^*U(F(\bs{\mr{z}_r}))+EV(\bs{\mr{z}_r}))$  and hence $BUF+CV=F(C^*UF+EV)$ or $(BUF+CV)(C^*UF+EV)^{-1}=F$ as $\{\boldsymbol{\mathrm{z}_r},\,\,r=1,\,2,\,...,\,n\}$ is a basis of $\mathbb{C}^n$. Thus $F$ is a fixed point of $h$.\\
  Conversely, Let $F$ be a fixed point of $h$ and  ${\mu_r}'s$ be the eigenvalues corresponding to the eigenbasis   $\{\boldsymbol{\mr{z}_r},\,\,r=1,\,2,\,...,\,n\}$  of $C^*UF+EV$. This gives\\
  $\left[ {\begin{array}{cc}
   BU & CV \\
   C^*U & EV \\
  \end{array} } \right] \left[{\begin{array}{c}
  F(\bs{\mr{z}_r})\\
  \bs{\mr{z}_r}\\
  \end{array}}\right]=\left[{\begin{array}{c}
 BU(F(\bs{\mr{z}_r}))+CV(\bs{\mr{z}_r}) \\
  C^*U(F(\bs{\mr{z}_r}))+EV(\bs{\mr{z}_r})\\
  \end{array}}\right]$\\
  $=\left[{\begin{array}{c}
 F(C^*U(F(\bs{\mr{z}_r}))+EV(\bs{\mr{z}_r})) \\
  C^*U(F(\bs{\mr{z}_r}))+EV(\bs{\mr{z}_r})\\
  \end{array}}\right]$   = $\mu_r\left[{\begin{array}{c}
  F(\bs{\mr{z}_r})\\
  \bs{\mr{z}_r}\\
  \end{array}}\right],\,\,r=1,\,2,\,...,\,n$. This completes the proof.
  \end{proof}
  The above proposition tells that the fixed points which are obtained via eigenvectors form a subclass of collection of all the fixed points.\\
  
 In view of  Theorem \ref{c}, \textbf{\textit{Fixed points of an isometry} \pmb{$T \in \mathcal{G}$} \textbf{would mean fixed points of its identification in} \pmb{$Aut(\mathcal{B})$}} \textbf{\textit{and from now onwards, fixed points would mean the ones which are  obtained through eigenvectors only,  (cf. Proposition \ref{P}).}}\\  
 
 In order to investigate the fixed points of a non-unitary normal isometry, let us first examine its spectrum.
 \begin{lemma}\label{l2}
  For the matrix $S_i=\left[ {\begin{array}{cc}
   a_iI & I \\
   {\delta_i}^2I & a_iI \\
  \end{array} } \right]_{2k_i \times 2k_i}$ where $I$ is a $k_i \times k_i$ identity matrix, the eigenvalues are $a_i \pm \delta_i$.  The corresponding eigenspaces  are generated by the sets $\{(\Delta_{m}^{(j)})_{m=1}^{2k_i},\,\,\,\,j \in J_i\}$ and $\{(\Delta_{m'}^{(j)})_{m'=1}^{2k_i},\,\,\,\,j \in J_i\}$ where 
  \begin{align*}
\Delta_{m}^{(j)} =
\left\{
	\begin{array}{llll}
		1  & \mbox{if } m=k_i+j, \\
		0 & \mbox{if } m=k_i+p,\, p \in J_i \setminus \{j\}, \\
		 \dfrac{1}{\delta_i} & \mbox{if }  j=m, \\
		0 & \mbox{if }  m \in J_i \setminus \{j\}
	\end{array}
\right.
    \end{align*}
  and $\Delta_{m'}^{(j)}$ is obtained by replacing $\dfrac{1}{\delta_i}$ by $-\dfrac{1}{\delta_i}$.
 \end{lemma}
 \begin{proof}
 For $\lambda_i \in \mathbb{C}$, consider the expression
\[ \left[ {\begin{array}{cc}
  ( a_i-\lambda_i)I & I \\
   {\delta_i}^2I & (a_i-\lambda_i) I \\
  \end{array} } \right]\left[ {\begin{array}{c}
  r_j\\
   \end{array} } \right]_{2k_i \times 1} \equiv 0. \]
  This implies
  \begin{align*}
  (a_i-\lambda_i)r_j+r_{k_i+j}&=0,\\
  {\delta_i}^2r_j+(a_i-\lambda_i)r_{k_i+j}&=0,
      \end{align*}
  $j \in J_i$. On simplification, the above two equations yield $\lambda_i=a_i \pm \delta_i$ are the two eigenvalues of $S_i$ and general eigenvectors are 
 $ \left[ {\begin{array}{c}
  r_j\\
   \end{array} } \right]_{2k_i \times 1}$, respectively where for arbitrary values of $r_{k_i+j} \in \mathbb{C}$, $r_j=\pm \dfrac{r_{k_i+j}}{\delta_i}$, $j \in J_i$. The eigenspaces corresponding to the  eigenvalues $\lambda_i \pm \delta_i$ are 
   \[E_i=span\,\Bigg\{\Big(\underbrace{0,\,0,\,...,\,\underbrace{ \dfrac{1}{\delta_i}}_{j^{th}\,\,\text{place}},0,...,0}_{\text{first}\,\,k_i\,\,\text{entries}},\underbrace{0,...,\underbrace{1}_{k_i +j^{th}\,\,\text{place}},0,...,0}_{\text{next}\,\,k_i\,\,\text{entries}}\Big),\,\,\,\,j \in J_i\Bigg\}\]
 and 
 \[{E_i}'=span\,\Bigg\{\Big(\underbrace{0,\,0,\,...,\,\underbrace{- \dfrac{1}{\delta_i}}_{j^{th}\,\,\text{place}},0,...,0}_{\text{first}\,\,k_i\,\,\text{entries}},\underbrace{0,...,\underbrace{1}_{k_i +j^{th}\,\,\text{place}},0,...,0}_{\text{next}\,\,k_i\,\,\text{entries}}\Big),\,\,\,\,j \in J_i\Bigg\}\] respectively  which is expanded version of the conclusion of the statement of lemma. 
 \end{proof}

  \begin{proposition} \label{f}
   Let $T=\left[ {\begin{array}{cc}
   BU & CV \\
   C^*U & EV \\
  \end{array} } \right] \in \mathcal{G}$ be a non-unitary normal isometry. Then  $$\sigma(T)=\overset{l}{\underset{i=1}{\cup}}\{\lambda_{i_j}(a_i \pm \delta_i),\,\,j \in J_i\} \cup \sigma(U\restriction_{(ran\,C)^{\perp}}) \cup \sigma_{pt}(V\restriction_{{K'}^{\perp}})$$ where   ${\lambda_{i_j}}$, $j \in J_i$  are eigenvalues of $U\restriction_{M_i}$ (cf (\ref{eq:36})) for each $i \in I$. The  eigenvectors corresponding to the eigenvalues $\lambda_{i_j}(a_i \pm \delta_i)$ are  $\left(\pm\dfrac{1}{\delta_i}C(\boldsymbol{\mr{z_{i_j}}}),\bs{\mr{z_{i_j}}}\right)$, respectively for $j \in J_i$ and $i \in I$ where the collection $\{\bs{\mr{z_{i_j}}},\,\,j \in J_i\}$ is some orthonormal basis of $K_i$ for each $i \in I$.
  \end{proposition}
   \begin{proof}
   In view of Corollary \ref{e}, $\sigma(T)=\overset{l}{\underset{i=1}{\cup}}{\sigma_{pt}{(T_i)}} \cup \sigma(U\restriction_{(ran\,C)^{\perp}}) \cup \sigma_{pt}(V\restriction_{{K'}^{\perp}})$ where
   $\sigma(U\restriction_{(ran\,C)^{\perp}})$ and $\sigma_{pt}(V\restriction_{{K'}^{\perp}})$ are  contained in the unit circle. We will show that for every $i \in I$, $\sigma_{pt}{(T_i)}=\{{\lambda_{i_j}}(a_i \pm \delta_i),\,\,j \in J_i\}$. \\
   For $i \in I$, consider $T_i=T\restriction_{M_i \oplus K_i}$. From Corollary \ref{c2},  
\begin{align*}
 [T_i]_{{\beta_i}' \cup {\beta_i}}&=\left[ {\begin{array}{cc}
   a_iR_i & R_i \\
   {\delta_i}^2 R_i & a_i R_i \\
  \end{array} } \right]_{2k_i \times 2k_i}\\
  &= \left[ {\begin{array}{cc}
   a_i{[I]}_{k_i \times k_i} & {[I]}_{k_i \times k_i} \\
   {\delta_i}^2{[I]}_{k_i \times k_i} & a_i{[I]}_{k_i \times k_i} \\
  \end{array} } \right] \left[ {\begin{array}{cc}
   R_i & 0 \\
   0 & R_i \\
  \end{array} } \right]=S_i\,W_i
  \end{align*}
  where $S_i=\left[ {\begin{array}{cc}
   a_i{[I]}_{k_i \times k_i} & {[I]}_{k_i \times k_i} \\
   {\delta_i}^2{[I]}_{k_i \times k_i} & a_i{[I]}_{k_i \times k_i} \\
  \end{array} } \right]$  and $W_i=\left[ {\begin{array}{cc}
   R_i & 0 \\
   0 & R_i \\
  \end{array} } \right]$.
From the proof of Lemma \ref{l2}, it can be observed that $S_i$ is diagonalizable. It has eigenvalues  $a_i \pm \delta_i$ and  $E_i$ and ${E_i}'$ as corresponding eigenspaces. Also it is easy to see that $S_i$ and $W_i$ commute. This implies $W(E_i)=E_i$ and $W({E_i}')={E_i}'$.   Let $\left\{\left({c_{j_p}}^{(i)}\right)_{p=1}^{k_i},\,\,j \in J_i\right\}$ be an orthonormal eigenbasis of $R_i$ having $\lambda_{i_j}$ as eigenvalues, ${c_{j_p}}^{(i)} \in \mathbb{C}$. Hence  $\left\{\left(\dfrac{\left({c_{j_p}}^{(i)}\right)_{p=1}^{k_i}}{\delta_i},\left({c_{j_p}}^{(i)}\right)_{p=1}^{k_i}\right),\,\,j \in J_i\right\}$ and \\ $\left\{\left(-\dfrac{\left({c_{j_p}}^{(i)}\right)_{p=1}^{k_i}}{\delta_i}, \left({c_{j_p}}^{(i)}\right)_{p=1}^{k_i}\right),\,\,j \in J_i\right\}$ would serve as eigenbases of $W_i\restriction_{{E_i}}$ and $W_i\restriction_{{E_i}'}$  respectively with common eigenvalues as $\{{\lambda_{i_j}},\,\,j \in J_i\}$.  
As these are also eigenbases of $T_i \restriction_{E_i}$ and  $T_i \restriction_{{E_i}'}$, union of these eigenbases is eigenbasis of $[T_i]_{{\beta_i}' \cup {\beta_i}}$.   Let
\begin{eqnarray*}
\bs{\mr{z_{i_j}}}=\sum\limits_{p=1}^{k_i}{{c_{j_p}}^{(i)}}e_{i_p}
\end{eqnarray*}
where $\{e_{i_p},\,\,p \in J_i\}$ is as in (\ref{eq:24}). Hence $\{\bs{\mr{z_{i_j}}},\,\,j \in J_i\}$ is an orthonormal basis of $K_i$. Also 
\begin{eqnarray}\label{eq:38}
C(\bs{\mr{z_{i_j}}})=\sum\limits_{p=1}^{k_i}{{c_{j_p}}^{(i)}}\xi_{i_p}
\end{eqnarray}
by (\ref{eq:29}). This gives the eigenvectors corresponding to eigenvalues $\lambda_{i_j}(a_i \pm \delta_i)$
  as  $\left(\pm\dfrac{1}{\delta_i}C(\boldsymbol{\mr{z_{i_j}}}),\bs{\mr{z_{i_j}}}\right),\,\,\,\,j \in J_i$.
\end{proof}
 
   It is easy to observe that eigenvalues $\lambda_{i_j}(a_i \pm \delta_i)$ are distinct for distinct values of  $i$. Also $\sigma(T_i \restriction_{E_i}) \cap \sigma(T_i \restriction_{{E_i}'})=\phi$.
   
   \begin{remark}\label{R1}
   Notice that for a non-unitary normal isometry \\ $T=\left[ {\begin{array}{cc}
   BU & CV \\
   C^*U & EV \\
  \end{array} } \right] \in \mathcal{G}$,    $\left[{\begin{array}{c}
  h\\
  \bs{\mr{z}}\\
  \end{array}}\right] \in {(ran\,C)}^{\perp} \oplus {K'}^{\perp}$ is an eigenvector of  $T$ if and only if  there exists an eigenvalue $\mu$ such that $U(h)=\mu h$ and $V(\bs{\mr{z}})=\mu \bs{\mr{z}}$.
   \end{remark}
Let us conclude the preceding analysis by exhibiting some special fixed points of a non-unitary normal isometry.   
 \begin{theorem} \label{g}
 Let $T=\left[ {\begin{array}{cc}
   BU & CV \\
   C^*U & EV \\
  \end{array} } \right] \in \mathcal{G}$ be a non-unitary normal isometry such that $dim\,(ran\,C)=k$. Then it has exactly $2^k$ fixed points of the generic type (cf. Definition \ref{d2})  
  such that $\|F_\theta\|=1$ for all $\theta \in S$. 
 \end{theorem}
 \begin{proof}
 In view of Definition \ref{d2}, consider $2^k$ number of generic functions for $q=k$ as follows.  As $\sum\limits_{i=1}^{l}{k_i}=k$, re-define the $k$-tuple $(\epsilon_m)_{m=1}^{k}$ as $(\epsilon_{i_j})_{i\in I,\,\,j\in J_i}$, orthonormal basis $\{\bs{\mr{z_m}},\,\,m=1,\,2,\,...,\,k\}$ as $\{\bs{\mr{z_{i_j}}},\,\,i \in I,\,\,j\in J_i \}$ where the collection $\{\bs{\mr{z_{i_j}}}\}$ is as in Proposition \ref{f} and choose $\{\bs{\mr{z_s}},\,\,s=k+1,\,....,\,n\}$ to be  an orthonormal eigenbasis  of $V \restriction_{{K'}^{\perp}}$. Take $\{y_m,\,\,m=1,\,2,\,...,\,k\}=\left\{\dfrac{C(\bs{\mr{z_{i_j}}})}{\delta_i},\,\,i \in I,\,\,j\in J_i \right\}$. Note that $\dfrac{C(\bs{\mr{z_{i_j}}})}{\delta_i} \neq 0$ as $C\restriction_{K'}$ is a bijective operator. For generic functions defined in this way, Proposition \ref{f} tells that  $\left[\begin{array}{c}
 F_\theta(\bs{\mr{z_{i_j}}})\\
 \bs{\mr{z_{i_j}}}\\
 \end{array}\right]$ are eigenvectors of $T$ for each $\theta \in S$,  $i \in I$ and $j \in J_i$. Also it is easy to see that $\left[\begin{array}{c}
 0\\
 \bs{\mr{z}_{s}}\\
 \end{array}\right]$ are eigenvectors of $T$ for $s=k+1,\,...,\,n$. Thus in view of Proposition \ref{P}, $T$ has exactly $2^k$ fixed points of the generic type. Now we will see that $\|F_\theta\|=1$ for $\theta \in S$.\\
 Equations (\ref{eq:38}) and (\ref{eq:26}) yield that the collection $\left\{\dfrac{C(\bs{\mr{z_{i_j}}})}{\delta_i},\,\,i \in I,\,j\in J_i\right\}$ is an orthonormal set and hence forms an orthonormal basis of $ran\,C$. This makes the generic function ${F_\theta}\restriction_{K'}$ a norm preserving surjective linear map and hence $\|{F_\theta}\restriction_{K'}\|=1$. Also ${F_\theta}\restriction_{{K'}^{\perp}} \equiv 0$. Therefore $\|F_\theta\|=\|{F_\theta}\restriction_{K'}\|=1$.
\end{proof}
 
 Another way of obtaining fixed points is subject to some condition. Let $\sigma_{pt}(U \restriction_{{(ran\,C)}^{\perp}}) \cap \sigma_{pt}(V \restriction_{{K'}^{\perp}}) \neq \phi$ and  $\mu$ be a common eigenvalue. Define $F \in B(\mathbb{C}^n ,H)$ such that $F\restriction_{K'}={F_\theta}\restriction_{K'}$ where $F_\theta$ is as in Theorem \ref{g}. For ${K'}^{\perp}$, $F$ may take basis elements of the  eigenspace of $V \restriction_{{K'}^{\perp}}$ to the values in the  eigenspace of $U \restriction_{{(ran\,C)}^{\perp}}$ corresponding to the common eigenvalue $\mu$. Set $F$ to vanish on those eigenspaces whose corresponding eigenvalues have empty intersection with $\sigma_{pt}(U \restriction_{{(ran\,C)}^{\perp}})$. Such an $F$ is also a fixed point of non-unitary normal isometry by Remark \ref{R1} and  Proposition \ref{P}.\\

\subsection{Proof of Theorem \ref{y}}  
   \begin{lemma}\label{l1}
    Let an isometry $T=\left[ {\begin{array}{cc}
   BU & CV \\
   C^*U & EV \\
  \end{array} } \right] \in \mathcal{G}$ be such that \\ $dim\,(ran\,C)=k$ and it possesses exactly $2^q$ fixed points of generic type (cf. Definition \ref{d2}).  Then 
      $q=k$ and
     $\bs{\mr{z}_m} \in K_i$ and $y_m$ lies in corresponding $ M_i$ for every $m \in\{1,\,2,\,...,\,k\}$ and some $i \in I$ depending on $m$.

  \end{lemma}
  \begin{proof}
By hypothesis, $\left[\begin{array}{c}
F_\theta(\bs{\mr{z}_m})\\
\bs{\mr{z}_{m}}\\
\end{array}\right]$ are eigenvectors of $T$ for every $\theta \in S$ and $m=1,\,2,\,...,\,q$, i.e.   $\left[\begin{array}{c}
\epsilon_m\, y_m\\
\bs{\mr{z}_{m}}\\
\end{array}\right]=\left[\begin{array}{c}
\pm y_m\\
\bs{\mr{z}_{m}}\\
\end{array}\right]$ are eigenvectors of $T$.  Let $\lambda_m$ and ${\lambda_m}'$ be corresponding eigenvalues. Expanding and simplifying the so formed expressions, we get
   \begin{align}
     CV(\bs{\mr{z}_m})&=\dfrac{\lambda_m-{\lambda_m}'}{2} y_m \label{eq:11}\\
     BU(y_m)&=\dfrac{\lambda_m+{\lambda_m}'}{2} y_m \label{eq:9}\\
     EV(\bs{\mr{z}_m})&=\dfrac{\lambda_m+{\lambda_m}'}{2} \bs{\mr{z}_m},\,\,\,\,\text{and} \label{eq:10}\\
     C^*U(y_m)&=\dfrac{\lambda_m-{\lambda_m}'}{2} \bs{\mr{z}_m}.\nonumber
     \end{align}
     The above equations further simplify to
     \begin{align}
         (\lambda_m + {\lambda_m}')\,C^*B^{-1}(y_m)&=(\lambda_m-{\lambda_m}')\,\bs{\mr{z}_m},\,\,\,\,\text{and}\label{eq:5}\\
         (\lambda_m+{\lambda_m}')\,CE^{-1}(\bs{\mr{z}_m})&=(\lambda_m-{\lambda_m}')\,y_m\label{eq:6}. 
     \end{align}
       As $E$ is invertible,  (\ref{eq:10}) yields $\lambda_m \neq {-\lambda_m}'$. In view of (\ref{eq:33}) and (\ref{eq:32}), each $\bs{\mr{z}_m}$ is a linear combination of basis elements of $\overset{l}{\underset{i=1}{\oplus}}{K_i} \oplus {K'}^{\perp}$  and each $y_m$ is a linear combination of basis elements of $\overset{l}{\underset{i=1}{\oplus}}{M_i} \oplus {(ran\,C)}^{\perp}$ respectively. Thus  for each $m \in \{1,\,2,\,...,\,q\}$, let
     \begin{align*}
         \bs{\mr{z}_m}&= \sum _{i=1}^{l} \sum_{j=1}^{k_i}{r^{(m)}_{{i}_{j}}}\,e_{{i}_{j}}+\bs{\mr{w_m}},\,\,\,\,\bs{\mr{w_m}} \in {K'}^{\perp}, \,\,\text{(by}\,\,(\ref{eq:24}))\\
         y_m&= \sum _{i=1}^{l} \sum_{j=1}^{k_i}{s^{(m)}_{{i}_{j}}}\,\xi_{{i}_{j}}+h_m,\,\,\,\,h_m \in {(ran\,C)}^{\perp},\,\,\text{(by}\,\,(\ref{eq:27}))
     \end{align*}
     $r^{(m)}_{{i}_{j}},\,s^{(m)}_{{i}_{j}} \in \mathbb{C}$.\\ 
 Substituting the above values of $\bs{\mr{z}_m}$ and $y_m$ in (\ref{eq:5}) gives $\bs{\mr{w_m}}=0$, hence $\bs{\mr{z}_m} \in K'$. This also provides the following relation between the coefficients $r^{(m)}_{{i}_{j}}$ and $s^{(m)}_{{i}_{j}}$ 
     
     \begin{align} \label{eq:7}
        s^{(m)}_{{i}_{j}}=\dfrac{\lambda_m-{\lambda_m}'}{\lambda_m+{\lambda_m}'}\,\dfrac{a_i}{{\delta_i}^2}\,r^{(m)}_{{i}_{j}} 
    \end{align}
    using (\ref{eq:31}) and (\ref{eq:30}).
    Similarly, using the values of $\bs{\mr{z}_m}$, $y_m$ and (\ref{eq:7}) in (\ref{eq:6}) yield $h_m=0$, hence $y_m \in ran\,C$.  Comparing the coefficients in (\ref{eq:6}) and using (\ref{eq:7}) yield
    \begin{align} \label{eq:8}
        \left(\dfrac{1}{a_i}-{\left(\dfrac{\lambda_m - {\lambda_m}'}{\lambda_m + {\lambda_m}'}\right)}^2\,\dfrac{a_i}{{\delta_i}^2}\right)\,r^{(m)}_{{i}_{j}}=0
    \end{align}
    by (\ref{eq:29}) and (\ref{eq:28}).
     In view of (\ref{eq:7}), $\lambda_m={\lambda_m}'$ is a contradiction to the hypothesis that all ${y_m} 's$ are non-zero as $h_m=0$. So, $\lambda_m \neq \pm{\lambda_m}'$.
     
     We will see that $q=k$. Since $\{\bs{\mr{z}_m},\,\,m=1,\,2,\,...,q\} \subseteq K'$ is linearly independent, we have $q \leq k$ as $dim\,(K')=k$ by Remark \ref{r6}. Also, $\left[\begin{array}{c}
0\\
\bs{\mr{z_s}}\\
\end{array}\right]$, $s=q+1,\,...,\,n$ are eigenvectors of $T$ as fixed points obtained are of generic type. Let the corresponding eigenvalues be $\lambda_s$. Expanding the concerned expression gives $CV(\bs{\mr{z}_s})=0$. This means $dim\,(\text{ker}\,C)\,(=n-k) \geq n-q$, i.e. $q \geq k $ hence $q=k$. 
     
    Next claim is that each $\bs{\mr{z}_m},\,y_m$ lies in some $K_i,\,M_i$.  As $\bs{\mr{z}_m} \neq 0$, for every $m \in \{1,\,2,\,...,\,k\}$, there exist atleast one $i \in I$ and $j \in J_i$  such that $r^{(m)}_{{i}_{j}} \neq 0$ which means 
   $ \dfrac{1}{a_i}-{\left(\dfrac{\lambda_m - {\lambda_m}'}{\lambda_m + {\lambda_m}'}\right)}^2\,\dfrac{a_i}{{\delta_i}^2}=0$  or 
    \begin{eqnarray}\label{eq:17}
    \dfrac{\lambda_m - {\lambda_m}'}{\lambda_m + {\lambda_m}'}= \pm \dfrac{{\delta_i}}{a_i}
    \end{eqnarray}
   using (\ref{eq:8}) and hence 
    \begin{align} \label{eq:18}
        s^{(m)}_{{i}_{j}}=\pm\dfrac{1}{\delta_i}\,r^{(m)}_{{i}_{j}} 
    \end{align}
    by (\ref{eq:7}). Suppose there exists  $p \neq i \in \{1,\,2,\,...,\,l\}$ such that  $r^{(m)}_{{p}_{j}} \neq 0$ for some $j \in J_p$. This means  $\dfrac{1}{a_{p}}-{\left(\dfrac{\lambda_m - {\lambda_m}'}{\lambda_m + {\lambda_m}'}\right)}^2\,\dfrac{a_{p}}{{\delta_{p}}^2}=0$ which further implies $\dfrac{{\delta_i}^2}{(a_i)^2}=\dfrac{{\delta_{p}}^2}{(a_{p})^2}$, and $a_i=a_{p}$ by (\ref{eq:39}) giving a contradiction. Thus for each $m \in \{1,\,2,\,...,\,k\}$, there exists $i \in I$ such that 
    \begin{eqnarray*}
    \bs{\mr{z}_m}=  \sum\limits_{j=1}^{k_i}{r^{(m)}_{{i}_{j}}}\,e_{{i}_{j}} \in K_i 
    \end{eqnarray*}
     by (\ref{eq:24}). As (\ref{eq:7}) yields ${r_{i_j}}^{(m)}$ and ${s_{i_j}}^{(m)}$ vanish together,
    \begin{eqnarray}\label{eq:2.}
    y_m=\sum\limits_{j=1}^{k_i}{s^{(m)}_{{i}_{j}}}\,\xi_{{i}_{j}} \in M_i
    \end{eqnarray}
    using  (\ref{eq:27}). 
    \end{proof}
    
   \textbf{Proof of Theorem \ref{y}}
  Necessary part follows from Theorem \ref{g}. For the sufficient part we will show $T$ is normal by using Proposition \ref{P1} together with Lemma \ref{l1}. First we will show that the set $\{y_m,\,\,m=1,\,2,\,...,\,k\}$   is  linearly independent.  As each $\bs{\mr{z}_m}$ lies in exactly one $K_i$ by Lemma \ref{l1}, we have each $\bs{\mr{z}_m}$ is an eigenvector of $V$ by (\ref{eq:10}) and (\ref{eq:28}),  hence $K'$ reduces $V$. Now bijectivity of   $C\restriction_{K'}$  asserts linear independence of $\{y_m\}$'s using (\ref{eq:11}). Again using $\bs{\mr{z}_m}$ is an eigenvector of $V$ and (\ref{eq:10}), it follows that $EV=VE$. Similar argument applied  to subspaces of $ran\,C$ establishes $UB=BU$.  
  Now we need to show $UC=CV$.
 Using (\ref{eq:2.}) and (\ref{eq:18}),  $y_m=\pm\dfrac{1}{\delta_i} \sum\limits_{j=1}^{k_i}{r^{(m)}_{{i}_{j}}}\,\xi_{{i}_{j}}=\pm\dfrac{1}{\delta_i} C(\bs{\mr{z}_m})$ 
 or $C(\bs{\mr{z}_m})=\pm \delta_i\,y_m$ by (\ref{eq:29}). Now using  (\ref{eq:9}) and (\ref{eq:31}), $UC(\bs{\mr{z}_m})=\pm \delta_i  \left(\dfrac{\lambda_m+{\lambda_m}'}{2a_i}\right) \,y_m$ $=\left(\dfrac{\lambda_{m}-{\lambda_{m}}'}{2}\right)\, y_{m} =CV(\bs{\mr{z}_{m}})$ by (\ref{eq:17}) and (\ref{eq:11}). Hence $T$ is normal and by Theorem \ref{g}, its generic fixed points lie on $\partial{\mathcal{B}}$. This completes the proof.

\begin{remark}
When $dim\,(K)>1$, then there exist non-trivial orthogonal vectors amongst and between the collection of time-like and light-like vectors. For example, if $H=K=\mathbb{C}^2$ with standard orthonormal basis  $\{e_1,\,e_2\}$ then $(e_1,\,e_1)$ and $(e_2,\,e_2)$ are two light-like orthogonal vectors. $(0,\,e_1)$ and $(0,\,e_2)$ are two orthogonal time-like vectors and $(e_1,\,e_1)$ and $(0,\,e_2)$ are two orthogonal light-like and time-like vectors.
\end{remark}
The following example discussed in \cite[page 55]{FR}  shows that in general holomorphic automorphisms in a $C^*$-algebra may not possess fixed points.\\
\begin{example}
Consider the $C^*$-algebra of all complex valued continuous functions defined on $\overline{\Delta}=\{x \in \mathbb{C}\,\,\,\,|\,\,\,\,|x| \leq 1\}$ with $\|.\|_{\infty}$ norm. Let ${B}'$ be its open unit ball. Then the function $F \in Aut({B}')$ defined by $F(f)(x)=\dfrac{f(x)-\frac{1}{2}x}{1-\frac{1}{2}\overline{x}f(x)}$ doesnot have a fixed point in $\overline{{B}'}$.
\end{example}  
\subsection{Acknowledgement}
The authors acknowledge partial support from the TARE project grant\\
 TAR/2019/000379.  During the course of this work, Rachna acknowledges support from CSIR-SRF grant. 

Part of this work was carried out when two of the authors (R.A and M.M.M) visited IISER Mohali. They acknowledge hospitality and support from IISER Mohali during their stay there.

\bibliographystyle{amsplain}

\begin{thebibliography}{99}

\bibitem{GR}  Chen, S. S.;  Greenberg, L.: Hyperbolic spaces. Contributions to analysis (a collection of papers dedicated to Lipman Bers), pp. 49--87. Academic Press, New York, 1974. 


\bibitem{FV} Franzoni, T.; Vesentini, E.: Holomorphic maps and invariant distances. Notas de Matem$\acute{a}$tica [Mathematical Notes], 69. North-Holland Publishing Co., Amsterdam-New York, 1980. {\rm viii}+226 pp. ISBN: 0-444-85436-3 MR0563329

\bibitem{FR} Franzoni, T.: The group of holomorphic automorphisms in certain $J\sp{\ast} $-algebras. Ann. Mat. Pura Appl. (4) 127 (1981), 51--66. MR0633394





\bibitem{GOL}  Goldman, W. M.: {\it Complex hyperbolic geometry}, Oxford Mathematical Monographs, The Clarendon Press, Oxford University Press, New York, 1999. MR1695450

\bibitem{GP} Gongopadhyay, K.; Parker, J. R.; Parsad, S.: On the classifications of unitary matrices, Osaka J. Math. {\bf 52} (2015), no.~4, 959--991. MR3426624

\bibitem{GON} Gongopadhyay, K.; Kulkarni, R. S.: $z$-classes of isometries of the hyperbolic space. Conform. Geom. Dyn. 13 (2009), 91--109. MR2491719

 \bibitem{GW} Greenfield, S. J.; Wallach, N. R.:  Automorphism groups of bounded domains in Banach spaces, Trans. Amer. Math. Soc. {\bf 166} (1972), 45--57. MR0296359

\bibitem{HA} Harris, L. A.: Bounded symmetric homogeneous domains in infinite dimensional spaces, in {\it Proceedings on Infinite Dimensional Holomorphy (Internat. Conf., Univ. Kentucky, Lexington, Ky., 1973)}, 13--40, Lecture Notes in Math., Vol. 364, Springer, Berlin. MR0407330

\bibitem{KW}  Kaup, W.: \"{U}ber die Automorphismen Grassmannscher Mannigfaltigkeiten unendlicher Dimension, Math. Z. {\bf 144} (1975), no.~2, 75--96. MR0404712


\bibitem{KA}  Kaup, W.; Upmeier, H.: Banach spaces with biholomorphically equivalent unit balls are isomorphic, Proc. Amer. Math. Soc. {\bf 58} (1976), 129--133. MR0422704

\bibitem{MA}  Mackey, D. S., \ et al.: Structured tools for structured matrices, Electron. J. Linear Algebra {\bf 10} (2003), 106--145. MR2001979

 \bibitem{MO}  Mostafazadeh, A.: Pseudounitary operators and pseudounitary quantum dynamics, J. Math. Phys. {\bf 45} (2004), no.~3, 932--946. MR2036172
 
 \bibitem{MU}  Munshi, S.; Yang, R.: Self-adjoint elements in the pseudo-unitary group $\mathbf{U}(p,p)$, Linear Algebra Appl. {\bf 560} (2019), 100--113. MR3866547


\bibitem{NA} Nagy, B. S., \ et al.: {\it Harmonic analysis of operators on Hilbert space}, second edition, revised and enlarged edition, Universitext, Springer, New York, 2010. MR2760647

 \bibitem {NE} Neretin, Y. A.: {\it Lectures on Gaussian integral operators and classical groups}, EMS Series of Lectures in Mathematics, European Mathematical Society (EMS), Z\"{u}rich, 2011. MR2790054
 
 \bibitem{OS} Ostrovskii,  M.~I.;  Shulman, V.~S.\ and\ Turowska,  L.~B.: Unitarizable representations and fixed points of groups of biholomorphic transformations of operator balls, J. Funct. Anal. {\bf 257} (2009), no.~8, 2476--2496. MR2555010

\bibitem{PA} Parker, J. R.: "Notes on complex hyperbolic geometry." preprint (2003)

\bibitem{Po} Popescu, G.: Noncommutative hyperbolic geometry on the unit ball of $B( H)^n$. J. Funct. Anal. 256 (2009), no. 12, 4030--4070. MR2521919

\bibitem{PO}  Porteous, I. R.: {\it Clifford algebras and the classical groups}, Cambridge Studies in Advanced Mathematics, 50, Cambridge University Press, Cambridge, 1995. MR1369094.

\end{thebibliography}

\end{document}